\documentclass[a4paper,UKenglish,cleveref,numberwithinsect]{lipics-v2021}

\pdfoutput=1 

\title{The Leibniz adjunction in homotopy type theory,\\ with an application to
  simplicial type theory}

\titlerunning{The Leibniz adjunction in HoTT, with an application to STT}

\author{Tom {de Jong}}{School of Computer Science, University of Nottingham, United Kingdom \and Institute for Computing and Information Sciences, Radboud University, Nijmegen, The Netherlands \and \url{http://www.tdejong.com}}{tom.dejong@ru.nl}{https://orcid.org/0000-0003-1585-3172}{Royal Society [URF{\textbackslash}R1{\textbackslash}191055, URF{\textbackslash}R{\textbackslash}241007]}

\author{Nicolai Kraus}{School of Computer Science, University of Nottingham, United Kingdom \and \url{https://nicolaikraus.github.io}}{nicolai.kraus@nottingham.ac.uk}{https://orcid.org/0000-0002-8729-4077}{Royal Society [URF{\textbackslash}R1{\textbackslash}191055, URF{\textbackslash}R{\textbackslash}241007], Engineering and Physical Sciences Research Council [EP/Z000602/1]}

\author{Axel Ljungstr\"om}{Computer \& Information Sciences, University of Strathclyde, United Kingdom \and \url{https://aljungstrom.github.io}}{axel.ljungstrom@strath.ac.uk}{https://orcid.org/0000-0001-6946-0775}{Knut and Alice Wallenberg Foundation's program for mathematics}

\authorrunning{T. de Jong and N. Kraus and A. Ljungstr\"om}

\Copyright{Tom de Jong and Nicolai Kraus and Axel Ljungstr\"om}



\category{} 

\hypersetup{hypertexnames=false}

\newcommand{\repo}{https://github.com/aljungstrom/Leibniz}
\newcommand{\baseurl}{https://arxiv.org/src/2601.21843v4/anc/html/Paper.html}
\newcommand{\htmlurl}{https://arxiv.org/src/2601.21843v4/anc/html/README.html}

\newcommand{\swhid}{swh:1:dir:7991df085271e86778a86793b52ebc9f7f3bae01}
\newcommand{\swhfull}{\swhid;origin=https://github.com/aljungstrom/Leibniz}

\supplement{Formalization in Cubical Agda}
\supplementdetails[swhid={\swhfull}]{Repository}{\repo}
\supplementdetails{HTML rendering}{\baseurl}

\acknowledgements{We thank Christian Sattler for an interesting discussion on
  the Leibniz adjunction and Aref Mohammadzadeh for pointing out some typos.}

\nolinenumbers 

\EventEditors{Corina C\^{\i}rstea and Yannick Forster}
\EventNoEds{2}
\EventLongTitle{35th EACSL Annual Conference on Computer Science Logic (CSL 2027)}
\EventShortTitle{CSL 2027}
\EventAcronym{CSL}
\EventYear{2027}
\EventDate{January 26--29, 2027}
\EventLocation{Brighton, UK}
\EventLogo{}
\SeriesVolume{403}
\ArticleNo{5}

\usepackage[dvipsnames]{xcolor}
\usepackage{stmaryrd}
\usepackage{mathtools}
\usepackage{tikz-cd}
\pgfdeclarelayer{background}
\pgfsetlayers{background,main}

\newcommand*{\pbcorner}[1][dr]{\ar[#1,phantom,"\lrcorner" , very near start]}
\newcommand*{\pocorner}[1][dr]{\ar[#1,phantom,"\ulcorner" , very near end]}

\newcommand{\intmeet}{\pmb\wedge}
\newcommand{\intjoin}{\pmb\vee}
\newcommand{\colonequiv}{\mathrel{\vcentcolon\mspace{-1mu}\equiv}}
\newcommand{\defeq}{\colonequiv}
\newcommand{\equivcolon}{\mathrel{\equiv\mathrel{:}}}

\newcommand{\Ltimes}{\mathbin{\widehat\times}}
\newcommand{\Lexp}{\mathbin{\pitchfork}}
\newcommand*{\lprod}{\Ltimes}
\newcommand*{\lexp}{\Lexp}
\newcommand{\N}{\mathbb N}
\newcommand{\Fin}[1]{\mathsf{Fin}\, #1}

\DeclareMathOperator{\ap}{ap}

\DeclareMathOperator{\refl}{refl}

\DeclareMathOperator{\fiber}{fib}
\DeclareMathOperator{\const}{const}

\DeclareMathOperator{\pr}{pr}
\newcommand*{\fst}{\pr_1}
\newcommand*{\snd}{\pr_2}

\newcommand*{\U}{\mathcal U}
\newcommand*{\C}{\mathcal C}
\newcommand*{\D}{\mathcal D}

\DeclareMathOperator{\Map}{\textup{\textsf{Map}}}
\DeclareMathOperator{\Fam}{\textup{\textsf{Fam}}}

\newcommand*{\join}{\mathbin{\ast}}

\newcommand{\sbullet}{\raisebox{1.15pt}{\scalebox{0.5}{$\bullet$}}}

\newcommand{\SigmaT}[2]{\textstyle\sum\pa*{#1}\,{#2}}
\newcommand{\PiT}[2]{\textstyle\prod\pa*{#1}\,{#2}}
\newcommand{\lambdadot}[2]{\lambda{#1}.\,{#2}}

\DeclareMathOperator{\id}{id}
\DeclareMathOperator{\op}{op}
\DeclarePairedDelimiter{\pa}{(}{)}
\newcommand{\orth}{\mathrel{\perp}}

\newcommand{\qedNoProof}{\hfill\qedsymbol}

\newcommand{\formalized}{{\color{NavyBlue!75!White}{\raisebox{-0.5pt}{\scalebox{0.8}{\faCog}}}}}
\newcommand{\flinkurl}[1]{\href{#1}{\formalized}}
\newcommand{\flink}[1]{\flinkurl{\baseurl\##1}}

\newenvironment{fequation}[1]
{
  \newtagform{formalized-eq}{(\flink{#1}\,}{)}%
  \usetagform{formalized-eq}%
  \begin{equation}%
}
{
  \end{equation}%
  \usetagform{default}\ignorespacesafterend%
}

\begin{document}

\maketitle

\begin{abstract}
  \emph{Simplicial type theory} extends homotopy type theory and equips types
  with a notion of directed morphisms. A \emph{Segal type} is defined to be a
  type in which these directed morphisms can be composed. We show that all
  higher coherences can be stated and derived if simplicial type theory is taken
  to be homotopy type theory with a postulated interval type. In technical
  terms, this means that if a type has unique fillers for $(2,1)$-horns, it has
  unique fillers for all inner $(n,k)$-horns. This generalizes a result of Riehl
  and Shulman for the case $n = 3$, $k \in \{1,2\}$.

  Our main technical tool is the Leibniz adjunction: the pushout-product is left
  adjoint to the pullback-hom in the wild category of types. While this
  adjunction is well known for ordinary categories, it is much more involved for
  higher categories, and the fact that it can be proved for the wild category of
  types (a higher category without stated higher coherences) is non-trivial. We
  make profitable use of the equivalence between the wild category of maps and
  that of families.

  We have formalized the results in Cubical Agda.
\end{abstract}

\section{Introduction}

Riehl and Shulman have introduced \emph{simplicial type theory (STT)} as a
synthetic (bespoke) framework
for higher categories~\cite{RS2017}.
It extends homotopy type theory, where types are equipped with invertible
(higher) paths, making them behave like $\infty$-groupoids rather than
categories.
Therefore, STT additionally equips types with a notion of \emph{directed}
morphisms via an interval.
We naturally wish to compose these directed morphisms, and those types where
this is possible are referred to as \emph{Segal types}.
Only having composition is not enough for a well behaved theory of higher
categories: we want composition to be associative and for identities to be
neutral with respect to composition.
For ordinary 1-categories, that would be sufficient, since equations are mere
properties, but for higher categories, equations carry genuine data and there
are (infinitely many) additional coherences to be fulfilled.
A simple example illustrating this is the requirement that the different ways of
rewriting \({\id} \circ (f \circ \id)\), illustrated below, are in fact the
same, up to a higher morphism.
\[
  \begin{tikzcd}[column sep=0mm, row sep=4mm]
    {\id} \circ (f \circ \id) \ar[rr,Rightarrow]
    \ar[d,Rightarrow]
    &&  ({\id} \circ f) \circ \id
    \ar[d,Rightarrow] \\
    f \circ \id \ar[dr,Rightarrow] && {\id} \circ f \ar[dl,Rightarrow] \\
    & f
  \end{tikzcd}
\]

Informally, the \textbf{first main result} of this paper
(\cref{thm:inner-horns-are-anodyne}) states that for Segal types, \emph{all}
higher coherences can be derived.
Simplicial techniques give us a precise and combinatorial framework for stating
higher coherences.
The 2-simplex~\(\Delta^2\) is pictured as a filled-in triangle with directed
edges while the (2,1)-horn \(\Lambda^2_1\) is obtained by omitting the filling and the
hypotenuse:
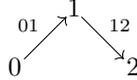
\[
  \Delta^2 = \begin{tikzcd}[column sep=.3cm, row sep=.3cm,%
    execute at end picture={
      \begin{pgfonlayer}{background}
        \foreach \Name in  {A,B,C}
        {\coordinate (\Name) at (\Name.center);}
        \fill[dashed,gray!40]
        (A) -- (B) -- (C) -- cycle;
      \end{pgfonlayer}
    }
    ]
    & |[alias=A]|1 \ar[dr,"12",shorten=-5pt] \\
    |[alias=C]|0 \ar[ur,"01",shorten=-5pt] \ar[rr,"02"',shorten=-3pt] & & |[alias=B]|2
  \end{tikzcd}
  \qquad\text{and}\qquad
  \Lambda^2_1
  =
  \begin{tikzcd}[column sep=.3cm, row sep=.3cm]
    & 1 \ar[dr,"12",shorten=-5pt] \\
    0 \ar[ur,"01",shorten=-5pt] & & 2
  \end{tikzcd}
\]
The idea is that a map \(\Lambda^2_1 \to X\) picks out two composable morphisms
in \(X\), by \(01 \mapsto f\) and \(12 \mapsto g\), say, and if we can extend it to a map \(\Delta^2 \to X\), we can take their composition \(g \circ f\) to be the image of the edge \(02\).
The filling, a $2$-dimensional cell, then witnesses that the triangle commutes up to this higher cell.
If such an extension is unique (in the homotopy type theory sense), then we say that \(X\)
has unique fillers for \(\Lambda^2_1\)-horns, and we call $X$ a \emph{Segal type}.
Of course, one can similarly consider higher-dimensional simplices~\(\Delta^n\)
and (inner or outer) horns \(\Lambda^n_k\), for \(n \in \mathbb N\) and \(0 \leq k \leq n\), which encode the higher coherences of the composition.
\cref{thm:inner-horns-are-anodyne} can then be phrased more precisely as: any
Segal type has unique fillers for all such inner horns.
This theorem generalizes~\cite[Proposition~5.21]{RS2017} of Riehl and Shulman
where they consider \(n = 3\) and \(k \in \{1,2\}\), and is a version of
Lurie's~\cite[Corollary 2.3.2.2]{LurieHTT} in simplicial type theory.

The \textbf{second contribution} lies in our approach and technical development.
The idea, going back to Joyal (cf.~just before Proposition~7.6 in
\cite{JoyalTierney2006} where \cite{Joyal2008} is also cited), is to reduce
the situation to the abstract framework provided by the orthogonality and
Leibniz calculus~\cite{riehl2014theory}.
Our main technical tool is then that the pushout-product is left adjoint to the
pullback-hom (\cref{Leibniz-adjunction-main-thm:alt}) in the wild
category of types.
While this result is well known for 1-categories, a naive approach to proving
this for the wild category of types quickly becomes intractable due to the
involved bookkeeping of coherences.
Thankfully, as a consequence of the univalence axiom, functions and dependent
type families are equivalent.
This yields an equivalence between the wild category of maps and the wild
category of families, and we illustrate how the latter is a useful substitute
for the former.

\subsection{Foundational setup}\label{sec:foundational-setup}

Instead of working with Riehl--Shulman's simplicial type theory as introduced
in~\cite{RS2017}, we follow the approach of
Gratzer, Weinberger and Buchholtz~\cite{GWB2024,GWB2025} (in line
with Pugh~and~Sterling~\cite{PughSterling2025}) and work in plain
homotopy type theory with a postulated interval type.
In contrast, in the Riehl--Shulman framework, the interval is a concept on the
meta-level. The advantage of their approach is that one gets additional
definitional equalities, but the advantage of the approach followed here is
the possibility of formalizing everything in an
established proof assistant like (Cubical) Agda (cf.~\cref{sec:formalization}).

Compared to~\cite{GWB2024,GWB2025}, our assumptions on the interval type are
rather minimal: we only require it to be a poset that forms a bounded
distributive lattice. In particular, we do not require the order to be total,
i.e.\ we do not require $x \leq y \vee y \leq x$ for all $x$ and $y$.
Further, we do not assume the modalities of~\cite{GWB2024,GWB2025}.

\subsection{Related work}
A particular instance of what later became known as the Leibniz adjunction was
used by Joyal and Tierney~\cite{JoyalTierney2006,Joyal2008} to prove closure
properties of classes of maps in simplicial sets.
We mimic this but the structures we work with are not strict
1-categories. Our approach to the resulting complications is briefly outlined in
the introduction (see the \emph{second contribution} paragraph).
The name \emph{Leibniz adjunction} is due to Riehl and Verity,
see~\cite[Observation~4.2]{riehl2014theory}, who explain the similarity of the
adjunction formula to \emph{``Leibniz's formula for the boundary of a product of
  two polygons, or indeed his formula for the derivative of a product of
  functions''.}

In addition to the work discussed above~\cite{RS2017,GWB2024,GWB2025}, we
remark that various closure properties of left/right orthogonal maps
(cf.~\cref{sec:orthogonality}) have previously been formalized in Agda: first in
agda-unimath~\cite{agda-unimath:orthogonality} by~Bakke, and later by Toth
in~\cite{Toth2025}.
We note that the former does not include that left orthogonal maps are closed
under pushout-products (\cref{lem:left-orth-pushout-products}), which we obtain
as an application of the Leibniz adjunction.
The latter does contain this result, but uses a different proof method, and did
not contain the Leibniz adjunction.
Following our approach, Bakke has since formalized a dependent Leibniz
adjunction as part of Toth's Agda development for synthetic
mathematics~\cite{Bakke2026}.

\subsection{Formalization}\label{sec:formalization}
All our results are formalized using Cubical Agda~\cite{CubicalAgda}.
Our formalization builds on the cubical library~\cite{cubical} and type checks
using Agda~2.8.0 and Agda~2.9.0.
The code repository is available at \url{\repo} and archived on SoftwareHeritage as
\href{https://archive.softwareheritage.org/\swhfull}%
{\texttt{\swhid}}.

The HTML rendering of the code can be found at
\url{\htmlurl}.
The file \href{\baseurl}{\texttt{Paper.html}} acts as an
interface between the formal development and the paper. In particular, (nearly)
all environments in this paper are marked with a \formalized{} symbol which is a
clickable link to the corresponding formalized statement in that file.

\section{Preliminaries}
At times we find it convenient to abbreviate \(\SigmaT{x : X}{Y\,x}\) as \(\sum Y\).
When used in function types, our convention is that \(\sum Y \to A\) should be
read as \((\sum Y) \to A\).

The following two basic, but extremely useful equivalences
(cf.~\cite[p.~3]{deJong2025}) will be frequently employed in various proofs in
\cref{sec:Leibniz-adjunction}, often without explicitly mentioning their names.

\begin{description}
\item[Contractibility of singletons]\label{proj-equivalences} For any type \(X\)
  and \(x : X\), the type \(\SigmaT{y : X}{x = y}\) is contractible (i.e.\ it is
  equivalent to the unit type) and hence the two projection maps from
  \(\SigmaT{x : X}{\SigmaT{y : X}{x = y}}\) to \(X\) are equivalences.
\item[Distributivity of \(\prod\) over \(\sum\)] For a type \(A\) and
  dependent types \(B\) and \(Y\) over \(A\) and \(\SigmaT{a : A}{B(a)}\),
  respectively, the map
  \begin{align*}
    \PiT{a : A}{\SigmaT{b : B(a)}Y(a,b)}
    &\to \SigmaT{f : \PiT{a : A}{B(a)}}{\PiT{a : A}Y(a,f(a))} \\
    \alpha &\mapsto (\lambdadot{a}{\fst(\alpha(a)}),\lambdadot{a}{\snd(\alpha(a))})
  \end{align*}
  is an equivalence with inverse \((f,p) \mapsto \lambdadot{a}(f(a),p(a))\).

  This, and especially its nondependent version
  \[
    \PiT{a : A}{\SigmaT{b : B}{Y(a,b)}} \to \SigmaT{f : A \to B}{\PiT{a :
        A}{Y(a,f(a))}},
  \] is traditionally called the ``type theoretic axiom of choice'', but this
  name is misleading because there is no choice involved, see also~\cite[pp.~32
  and 104]{HoTTBook}.
\end{description}

\section{The Leibniz Adjunction for Types}\label{sec:Leibniz-adjunction}

In \cite[Section~4]{riehl2014theory}, Riehl and Verity describe the
\emph{Leibniz construction}. Given a functor
$F : \mathcal K \times \mathcal L \to \mathcal M$, with $\mathcal M$ having all
pushouts, it yields a functor $\widehat{F}$ on the arrow categories,
$\widehat{F} : \mathcal K^I \times \mathcal L^I \to \mathcal M^I$.
If we are instead given a functor
$G : \mathcal K^{\op} \times \mathcal L \to \mathcal M$ that is contravariant in
its first argument, then, assuming \(\mathcal M\) has all pullbacks, we apply
the Leibniz construction to
\(G^{\op} : \mathcal K \times \mathcal L^{\op} \to \mathcal M^{\op}\).

We only apply this construction to one particular instance, namely the wild
category of types (recalled below), and to two particular (wild) functors, namely
\(F \defeq (-)\times(-)\) and \(G \defeq (-)\to(-)\).
The Leibniz construction then yields two wild functors on the wild category of
maps, namely the \emph{pushout-product} \({\widehat F \equivcolon (-)\lprod(-)}\)
and \emph{pullback-hom} \(\widehat G \equivcolon (-)\lexp(-)\), that we describe
explicitly in~\eqref{Leibniz-product}~and~\eqref{Leibniz-exponential}
respectively.

The main result of this section is \cref{Leibniz-adjunction-main-thm:alt} (and
its corollary) which can be seen as a version of Riehl and Verity's
1-categorical adjunction \cite[Lemma~4.10]{riehl2014theory}
internal to homotopy type theory.

For us, the significance of the Leibniz adjunction goes back to reasoning about
fillers for certain horns, as described in the introduction. More specifically,
as explained in \cref{sec:orthogonality}, these fillers are special cases of the
concept of orthogonality, and two maps are orthogonal if their pullback-hom is
an equivalence (\cref{lem:equivs-are-orth}). Thus, the Leibniz adjunction
becomes a tool for reasoning about orthogonality, and hence about fillers.

\subsection{Wild categories and functors}

A naive translation of the definition of a (1-)category into homotopy
type theory results in the notion of a \emph{wild category}.
A wild category consists of
a type of objects and a type family of arrows with identities and an
associative composition operation for which the identities are left and right
neutral, as usual.
This notion is given the adjective ``wild'' because the types of arrows are not
restricted to be sets, nor do we require higher coherences on the identities and
composition~\cite{CapriottiKraus2017}.
The paradigmatic examples is the \emph{wild category of types} which has types
in a universe \(\U\) as objects and functions between such types as arrows.
The required identifications hold definitionally.

Similarly, we can naively translate the notion of a functor to arrive at the
definition of a \emph{wild functor} between two wild categories, which consists of a function between the types of objects, a family of functions between the types of arrows, and identifications witnessing that the latter maps identities to identities and compositions to compositions.
The Yoneda lemma extends to wild categories. In particular, we have a
contravariant wild functor from any wild category \(\C\) to the wild category of
types, given by \(A \mapsto \C(-,A)\), and two objects \(A\)~and~\(B\) of~\(\C\)
are isomorphic if and only if we have equivalences between \(\C(X,A)\) and
\(\C(X,B)\), natural in \(X\).

The following definition is in line with that of~\cite[Lemma~11]{KrausVonRaumer2019}, where the adjective \emph{naive} is dropped:
\begin{definition}[\flink{Definition-3-1} Naive isomorphism of wild categories]
	\label{def:wild-iso}
	A \emph{naive isomorphism between wild categories} \(\C\) and \(\D\) is a wild
	functor~\(F\) from \(\C\)~to~\(\D\) such that \(F\) is an equivalence on
	objects and on arrows.
\end{definition}

\begin{remark}[\flink{Remark-3-2}]
	The concepts of wild categories, wild functors, and naive isomorphisms are all ill-behaved \cite[\S 3.2 of the arXiv version]{kraus:inftyCwFs}, and may feel conceptually wrong.
	Nonetheless, they are often (as in the current paper) useful tools, and wild category theory has been studied in the literature, for example by Chen~\cite{chen20252coherentinternalmodelshomotopical}, Hart--Favonia~\cite{hart:LIPIcs.CSL.2025.46}, Christensen--Scoccola~\cite{ChristensenScoccola2023Hurewicz}, and Harington--Mimram~\cite{harington:entics2024}.
	A wild category should be thought of as an approximation to an $(\infty,1)$-category;
	it can be described as an $(\infty,1)$-category without all coherence data above the level of associativity.
	However, the ``operation'' of removing coherence data does not respect equivalences; in other words, two equivalent definitions of a concept can become different once coherence above a certain level is removed.

	This also means that the terminology itself is not completely unambiguous.
	For example, it could be argued that a wild functor should come with components witnessing that the identifications for associativity and identity equations are preserved. The notion we use, which follows the literature, does not include these components. As the concepts under consideration are merely tools, this ambiguity does not influence the final results.

	Nevertheless, care is required when the tools are used.
	A concrete shortcoming of \cref{def:wild-iso} is that naively isomorphic wild categories are, in general, not equal.
	This is because the data witnessing the equations of a wild category is
	not necessarily preserved by a naive isomorphism of wild categories.
	%
	In the formalization, we found it convenient to additionally consider
	\emph{wild preorders} which only keep the objects, hom
	types, identities and composition parts of a wild category, thus dropping the
	equations.
	Thanks to univalence, if we extend the notion of naive isomorphism from wild
	categories to wild preorders, naive isomorphism again coincides with equality.
	The upshot of this fact is that we can transfer properties of wild
	categories along naive isomorphisms, as long as these properties do not
	explicitly mention equations, which is the case for all key properties proved
	in this paper. \lipicsEnd
\end{remark}

\subsubsection{The wild category of maps \texorpdfstring{(\flink{3-1-1})}{}}

We define the wild category \(\Map\) of maps,
which is the wild version of the arrow category of the (higher) category of types and functions.
Spelled out, an object is a triple
\(({A : \U}, {B : \U},{f : A \to B})\), often abbreviated to just \(f\).
An~arrow in \(\Map(f,g)\) is given by (the dashed data in) the type of
commutative squares shown below.
\[
  \begin{tikzcd}
    A \ar[r,dashed,"u"] \ar[d,"f"]
    & X \ar[dl,dashed,"\alpha",Rightarrow] \ar[d,"g"] \\
    B \ar[r,dashed,"v"] & Y
  \end{tikzcd}
\]
Spelled out, this data consists of
\begin{gather*}
  u : A \to X, \quad v : B \to Y, \quad \alpha : g \circ u = v \circ f.
\end{gather*}
The identity at \(f : A \to B\) is simply \((\id_A,\id_B,\refl)\).
The composition of two commutative squares is as follows:
\[
  \begin{tikzcd}
    A \ar[r,"u"] \ar[d,"f"]
    & X \ar[dl,"\alpha",Rightarrow] \ar[d,"g"] \\
    B \ar[r,"v"] & Y
  \end{tikzcd}
  \quad\text{and}\quad
  \begin{tikzcd}
    X \ar[r,"m"] \ar[d,"g"]
    & W \ar[dl,"\beta",Rightarrow] \ar[d,"h"] \\
    Y \ar[r,"n"] & Z
  \end{tikzcd}
  \quad\text{compose to}\quad
  \begin{tikzcd}
    A \ar[r,"m \circ u"] \ar[d,"f"]
    & W \ar[dl,"\alpha\star\beta",Rightarrow] \ar[d,"h"] \\
    B \ar[r,"n\circ v"] & Z
  \end{tikzcd}
\]
where \(\alpha\star\beta\) is the composition of \( \ap_{(-) \circ u}\beta\) and
\(\ap_{n \circ (-)}\alpha\), where we use that associativity of function
composition holds definitionally.
It is straightforward but somewhat tedious to construct the identifications required for
this to form a wild category.

If one spells out the Leibniz construction for the cartesian product on types,
then one arrives at the following.
The \emph{pushout-product} of two maps \(f : A \to B\) and \(g : X \to Y\) is
the induced (dashed) map \(f \lprod g\) in the following diagram.
\begin{fequation}{dagger}\label{Leibniz-product}\tag{\(\dagger\)}
  \begin{tikzcd}[column sep=.5cm]
    A \times X \pocorner
    \ar[r,"f\times\id_X"]
    \ar[d,"{\id_A}\times g"]
    & B \times X \ar[d]
    \ar[ddr,bend left,"{\id_B} \times {g}"]
    \\
    A \times Y \ar[r]
    \ar[drr,bend right=20,"f \times \id_Y"]
    & (B \times X) +_{A \times X} (A \times Y)
    \ar[dr,dashed,"f\lprod g"]
    \\
    & &
    B \times Y
  \end{tikzcd}
\end{fequation}

If one spells out the Leibniz construction for taking function types, then one
arrives at the following.
The \emph{pullback-hom} of two maps \({f : A \to B}\) and \(g : X \to Y\) is the
induced (dashed) map \(f \lexp g\) in the following diagram.
\begin{fequation}{double-dagger}\label{Leibniz-exponential}\tag{\(\ddagger\)}
  \begin{tikzcd}[column sep=.5cm]
    X^B
    \ar[dr,dashed,"f \lexp g"]
    \ar[drr,bend left=20,"X^f"]
    \ar[ddr,bend right,"g^B"]
    \\
    & X^A \times_{B^A} Y^B \pbcorner
    \ar[d] \ar[r]
    & X^A \ar[d,"g^A"]
    \\
    & Y^B \ar[r,"Y^f"]
    & Y^A
  \end{tikzcd}
\end{fequation}
Here, \(g^{\sbullet}\) and \(\sbullet^f\) respectively denote postcomposition by
\(g\) and precomposition by \(f\).

The main result (\cref{Leibniz-adjunction-main-thm:alt}) of this overall section
is to establish the Leibniz adjunction: there is a natural equivalence
\(\Map({f \lprod g},h) \simeq \Map(f,{g \lexp h})\).
Unfolding the definition of the pushout-product and the
universal property of the pushout, an element of \(\Map({f \lprod g},h)\) is a
7-tuple consisting of three maps, three identifications and one coherence.
Similarly, an element of \(\Map(f,{g\lexp h})\) is also a 7-tuple and we can
define maps back and forth, but we found the tracking of all of the details
(even in a proof assistant) to be too much.
The fact that the identifications each depend on two of the three maps with the
coherence in turn depending on two of the identifications results in a very
challenging shortfall of modularity.
The lack of a direct proof motivates the rest of the section and our preference
for working with families instead of maps.

\subsubsection{The wild category of families \texorpdfstring{(\flink{3-1-2})}{}}
We define the wild category \(\Fam\) of type families. An object is
a pair \((A : \U, B : A \to \U)\).
An arrow in \(\Fam((A,B);(X,Y))\) is given by a pair consisting of a map \(m : A \to X\)
together with a dependent map \(\mu : \PiT{a : A}{\pa*{B(a) \to Y(m\,a)}}\), so
\[
  \Fam((A,B);(X,Y)) \defeq \SigmaT{m : A \to X}{\PiT{a : A}{(B\,a \to Y(m\,a))}}.
\]
The identity at \((A,B)\) is simply \((\id_A,\lambda a.\id_{Ba})\).
Given two arrows \((m,\mu)\) and \((n,\nu)\) in \(\Fam((X,Y);(C,D))\),
their composition is given by
\(
(n \circ m , \lambda a.\nu(m\,a) \circ \mu\,a).
\)
%
Unlike for \(\Map\), the identifications for a wild category hold definitionally
for \(\Fam\).

\subsubsection{The wild categories of maps and families are equivalent}

The fact that the following maps form an equivalence is a well known consequence
of the univalence axiom (see e.g.~\cite[Theorem~4.8.3]{HoTTBook}) and lies at
the heart of this section.
\begin{equation}\label{fiber-correspondence}\tag{\(\ast\)}
  \begin{aligned}
    \SigmaT{A : \U}{\SigmaT{B : \U}{(A \to B)}}
    &\simeq
    \SigmaT{X : \U}{(X \to \U)} \\
    (A,B,f) &\xmapsto{\chi} (B,b \mapsto \fiber_fb) \\
    \pa*{\SigmaT{x : X}{\phi\,x},X,\fst} &\mapsfrom (X,\phi)
  \end{aligned}
\end{equation}
Here we recall that \(\fiber_f b\) is the \emph{fiber} of \(f\) at \(b\), i.e.\
the type \(\SigmaT{a : A}{f\,a = b}\).
We will often write \(\chi \, f\) instead of \(\chi(A,B,f)\), leaving the domain
and codomain of~\(f\) implicit.

\begin{proposition}[\flink{Proposition-3-3}]\label{Map-Fam-equivalent}
  The equivalence \(\chi\) extends to an isomorphism of the wild categories
  \({\Map}\) and \({\Fam}\).
\end{proposition}
\begin{proof}
  We show that \(\chi^{-1}\) extends to a wild functor that is an
  equivalence on arrows.
  Given \((m,\mu) : \Fam((A,B);(X,Y))\), we write
  \(\fst^B : \sum B \xrightarrow{\fst} A\) and
  \(\fst^Y : \sum Y \xrightarrow{\fst} X\), and define
  \[
    \chi^{-1}(m,\mu) : \Map(\fst^B,\fst^Y)
  \]
  to be the commutative square
  \[
    \begin{tikzcd}
      \sum B \ar[d,"\fst^B"] \ar[r,"{[m,\mu]}"]
      & \sum Y \ar[d,"\fst^Y"]
      \ar[dl,Rightarrow,"\refl", near start]
      \\
      A \ar[r,"m"] & X
    \end{tikzcd}
  \]
  where \([m,\mu]\) is given by \((a,b) \mapsto (m\,a,\mu\,a\,b)\).
  It is direct that this assignment preserves identities and composition.
  To see that it is an equivalence, we note that it fits into the following
  commutative diagram of canonical equivalences
  \[
    \begin{tikzcd}[column sep=-3em]
      \Fam((A,B);(X,Y)) \ar[r,"\chi^{-1}"]
      \ar[d,"\id"]
      & \Map(\fst^B,\fst^Y)
      \ar[dddl,"\simeq",bend left=35]
      \\
      \SigmaT{g : A \to X}{\PiT{a : A}{B\,a \to Y(g\,a)}}
      \\
      \SigmaT{g : A\to X}{\PiT{p : \textstyle\sum B}Y(g (\fst^B p))}
      \ar[u,"\simeq"]
      \\
      \SigmaT{g : A\to X}{\SigmaT{F : \textstyle\sum B \to \textstyle\sum Y}{%
          \pa*{{\fst^Y} \circ F = g \circ {\fst^B}}}}
      \ar[u,"\simeq"]
    \end{tikzcd}
  \]
  as is not hard to verify.
\end{proof}

\subsection{The pushout-product and pullback-hom on \texorpdfstring{\(\Fam\)}{Fam}}

We recall from~\cite[p.~198]{HoTTBook} that the \emph{join} \(A \join B\) of two
types \(A\) and \(B\) is defined as the pushout of the projections
\(A \xleftarrow{\fst} A \times B \xrightarrow{\snd} B\).
Note that the join admits a straightforward extension to a wild endofunctor on
the wild category of types.

As observed by Rijke in~\cite[Theorem~2.2]{Rijke2017}, we have
the following result relating the pushout-product with joins.
\begin{proposition}[\flink{Proposition-3-4}]\label{Leibniz-product-is-fiberwise-join}
  The pushout-product is the fiberwise join: for \(f : A \to B\) and
  \(g : X \to Y\), we have, for any \(b : B\) and \(y : Y\), an equivalence
  \[
    \fiber_{f \lprod g}(b,y) \simeq \fiber_f(b) \join \fiber_g(y).
  \]
\end{proposition}

\cref{Leibniz-product-is-fiberwise-join} justifies, via
\cref{preservation-of-Leibniz-product}, the following definition.
\begin{definition}[\flink{Definition-3-5} Pushout-product of families, \((A,B) \lprod (X,Y)\)]
  The \emph{pushout-product} of two type families \(B : A \to \U\) and
  \(Y : X \to \U\) is the type family
  \[
    (A,B) \lprod (X,Y) \defeq (A \times X,(a,x) \mapsto B(a) \join Y(x)).
  \]
\end{definition}

\begin{proposition}[\flink{Proposition-3-6}]\label{preservation-of-Leibniz-product}
  The maps \(\chi\) and \(\chi^{-1}\) from \eqref{fiber-correspondence} preserve
  the pushout-product.
\end{proposition}
\begin{proof}
  Since \(\chi\) is an equivalence, it suffices to show this for \(\chi\), but
  that follows immediately from \cref{Leibniz-product-is-fiberwise-join}.
\end{proof}

The pushout-product is actually a wild functor on \(\Fam\), and this fact is
used in stating~\cref{Leibniz-adjunction-main-thm:alt}.
\begin{proposition}[\flink{Proposition-3-7}]\label{Leibniz-product-functor}
  Given
  \((m,\mu) : \Fam((A,B);(X,Y))\) and \((n,\nu) : \Fam((A',B');(X',Y'))\),
  we construct
  \begin{align*}
    &(m,\mu) \lprod (n,\nu) : \Fam((A,B) \lprod (A',B');(X,Y) \lprod (X',Y')) \\
    &(m,\mu) \lprod (n,\nu) \defeq (m \times n, {(a,a') \mapsto \mu(a) \join \nu(a')}),
  \end{align*}
  where \(\join\) denotes the functorial action of the join.
  This makes \(\lprod\) a wild (binary) endofunctor on \(\Fam\).
\end{proposition}

\begin{definition}[\flink{Definition-3-8} Constant maps, \(\const f\)]
  For a map \(f : A \to B\), we write
  \[
    \const f\defeq \SigmaT{b : B}{\PiT{a : A}{f\,a = b}}
  \]
  for the type of witnesses that \(f\) is constantly a designated point.
\end{definition}

\begin{lemma}[\flink{Lemma-3-9}]\label{total-space-of-const}
  The total space of the \(\const\) family is equivalent to the codomain of the
  map in question, i.e.\ for types \(A\) and \(B\), we have
  \[
    \SigmaT{f : A \to B}{\const f} \simeq B,
  \]
  given by \((f,b,\sbullet) \mapsto b\).\qedNoProof
\end{lemma}

The pullback-hom on type families is more involved and can be calculated
by computing the fibers of the pullback-hom of two projection maps. Here,
we directly give the definition and then justify it via
\cref{preservation-of-Leibniz-exponential}.
\begin{definition}[\flink{Definition-3-10} Pullback-hom of families, \((A,B) \lexp (X,Y)\)]
  The \emph{pullback-hom} of two type families \(B : A \to \U\) and
  \({Y : X \to \U}\) is the type family indexed by
  \[
    \pa*{(A,B) \lexp (X,Y)}_0 \defeq
    \Fam((A,B);(X,Y))
  \]
  and which given an element \((m,\mu)\) of this type returns the type
  \[
    \pa*{(A,B) \lexp (X,Y)}_1(m,\mu) \defeq
    \PiT{a : A}{\const(\mu\,a)}.
  \]
\end{definition}

The following observation is not used anywhere but serves to illustrate the
connection to function types.
\begin{lemma}[\flink{Lemma-3-11}]\label{total-space-of-Leibniz-exponential}
  The total space of \((A,B) \lexp (X,Y)\) is equivalent to \(\pa*{\sum Y}^A\).
\end{lemma}
\begin{proof}
  Writing \(E \defeq (A,B) \lexp (X,Y)\), we note that
  \begin{align*}
    &\SigmaT{(m,\mu) : E_0}{E_1(m,\mu)} \\
    &\equiv
      \SigmaT{(m,\mu) : E_0}{\const(\mu\,a)} \\
    &\simeq
      \SigmaT{m : A \to X}{%
      \PiT{a : A}{%
      \SigmaT{\phi : B(a) \to Y(m\,a)}{\const(\phi)}}}
    \\
    &\simeq \SigmaT{m : A \to X}{\PiT{a : A}{Y(m\,a)}} &\text{(by \cref{total-space-of-const})} \\
    &\simeq \pa*{A \to \sum Y}. &&\qedhere
  \end{align*}
\end{proof}

\begin{proposition}[\flink{Proposition-3-12}]\label{preservation-of-Leibniz-exponential}
  The maps \(\chi\) and \(\chi^{-1}\) from \eqref{fiber-correspondence} preserve
  the pullback-hom.
\end{proposition}
\begin{proof}
  Since \(\chi\) is an equivalence, it suffices to prove this for its inverse~\(\chi^{-1}\).
  To this end, suppose we have type families \(B : A \to \U\) and
  \(Y : X \to \U\), and write
  \[
    p : \textstyle\sum B \to A \quad\text{and}\quad q : \textstyle\sum Y \to X
  \]
  for the first projections.
  Following~\eqref{Leibniz-exponential}, we consider
  \[
    \begin{tikzcd}[column sep=.5cm]
      {\pa*{\sum Y}}^A
      \ar[dr,dashed,"p \lexp q"]
      \ar[drr,bend left=20,"{\pa*{\sum Y}}^p"]
      \ar[ddr,bend right,"q^A"]
      \\
      &
      \bullet
      \pbcorner
      \ar[d] \ar[r]
      & {\pa*{\sum Y}}^{\sum B} \ar[d,"q^{\sum B}"]
      \\
      & X^A \ar[r,"X^p"]
      & X^{\sum B}
    \end{tikzcd}
  \]
  Let us abbreviate \((A,B) \lexp (X,Y)\) by \(E\) again.
  Then we notice that
  \begin{align*}
    \bullet &\equiv \SigmaT{l : \sum B \to \sum Y}{\SigmaT{r : A \to X}{q \circ l = r \circ p}} \\
            &\simeq \SigmaT{f : A \to X}{\PiT{a : A}{\PiT{b : Ba}{\SigmaT{x : X}{\SigmaT{y : Yx}{{f\,a} = x}}}}} \\
            &\simeq \Fam((A,B);(X,Y)) \\
            &\equiv E_0.
  \end{align*}
  Moreover, this equivalence fits in a diagram, where the leftmost map is given
  by \cref{total-space-of-Leibniz-exponential},
  \[
    \begin{tikzcd}
      \pa*{\sum Y}^A \ar[r,"p \lexp q"] \ar[d,"\simeq"] & \bullet \ar[d,"\simeq"] \\
      \sum E_1 \ar[r,"\fst"] & E_0
    \end{tikzcd}
  \]
  that commutes definitionally.
  Hence,
  \(\pa*{\pa*{\sum Y}^A,\bullet,p \lexp q} = \pa*{\sum E_1 , E_0 , \fst}\), as
  desired.
\end{proof}

\begin{proposition}[\flink{Proposition-3-13}]
  The pullback-hom on \(\Fam\) is functorial in both arguments:
  contravariant in the first and covariant in the second.
\end{proposition}
\begin{proof}
  Given \((m,\mu) : \Fam((A,B);(A',B'))\), we define an element of
  \[
    \Fam((A',B') \lexp (X,Y);(A,B) \lexp (X,Y))
  \]
  essentially via precomposition with \((m,\mu)\), i.e.\ it is given by
  \begin{align*}
    \Fam((A',B');(X,Y)) &\to \Fam((A,B);(X,Y)) \\
    (n,\nu) &\mapsto (n \circ m,
              \lambdadot{a : A}{\lambdadot{b : B\,a}{\nu\,(m\,a)\,(\mu\,a\,b)} })
  \end{align*}
  and by mapping
  \[
    c : \PiT{a' : A'}{\const(\nu\,a')} \equiv
    \PiT{a' : A'}{\SigmaT{y : Y(n\,a')}
      {\PiT{b' : B'\,a'}{\nu\,a'\,b' = y}}}
  \]
  to
  \[
    a \mapsto
    \pa*{\fst(c\,(m\,a))
      ,
      \lambdadot{b : B\,a}{\snd(c\,(m\,a))\,(\mu\,a\,b)}
    }
  \]
  which is of type
  \(\PiT{a : A}\const(\lambdadot{b : B\,a}{\nu\,(m\,a)\,(\mu\,a\,b)})\).

  Secondly, given \((m,\mu) : \Fam((X,Y);(X',Y'))\), we define an element of
  \[
    \Fam((A,B) \lexp (X,Y);(A,B) \lexp (X',Y'))
  \]
  essentially via postcomposition with \((m,\mu)\), i.e.\ it is given by
  \begin{align*}
    \Fam((A,B);(X,Y)) &\to \Fam((A,B);(X',Y')) \\
    (n,\nu) &\mapsto (m \circ n,
              \lambdadot{a : A}{\lambdadot{b : B\,a}{\mu\,(n\,a)\,(\nu\,a\,b)} })
  \end{align*}
  and given
  \( c : \PiT{a : A}{\const(\nu\,a)}
  \),
  we observe that if we write \(y_a\) for \(\fst(c(m\,a))\), then \(\nu\,a\) is
  constantly \(y_a\), so that \(\mu\,(n\,a)\,(\nu\,a\,-)\) is constantly
  \(\mu\,(n\,a)\,y\). This gives us an element of the desired type
  \(\PiT{a : A}\const(\lambdadot{b : B\,a}{\mu\,(n\,a)\,(\nu\,a\,b)})\).
\end{proof}

\subsection{The pushout-product and pullback-hom as adjoints}

\begin{lemma}[\flink{Lemma-3-14}]\label{Leibniz-product-associative-commutative}
  The pushout-product is associative and commutative, i.e.\ for any maps we have
  \[
    (f \Ltimes g) \Ltimes h = f \Ltimes (g \Ltimes h)
    \quad\text{and}\quad
    (f \Ltimes g) = (g \Ltimes f).
  \]
\end{lemma}
\begin{proof}
  By \cref{preservation-of-Leibniz-product} it suffices to check this equation
  in \(\Fam\), rather than in \(\Map\). But in \(\Fam\) this follows directly
  from the associativity of the cartesian product (\(\times\)) and the
  join~(\(\join\))~\cite[Proposition~1.8.6]{Brunerie2016}.

  Commutativity follows similarly, as the cartesian product and join are
  commutative. Alternatively, as in the formalization, one may check this
  directly in \(\Map\).
\end{proof}

\begin{lemma}[\flink{Lemma-3-15}]\label{join-and-constancy}
  For types \(X\), \(Y\) and \(Z\) we have equivalences
  \[
    (X \join Y \to Z)
    \;\simeq\; \SigmaT{f : X \to Z}{(Y \to \const f)}
    \;\simeq\; \SigmaT{g : Y \to Z}{(X \to \const g)}.
  \]
\end{lemma}
\begin{proof}
  We only spell out the details for one of the equivalences; the other is proved
  similarly. Starting with the universal property of the join, we get
  \begin{align*}
    (X \join Y \to Z)
    &\simeq
      \SigmaT{f : X \to Z}{\SigmaT{g : Y \to Z}{\PiT{(x,y) : X \times Y}{f\,x} = {g\,y}}} \\
    &\simeq
      \SigmaT{f : X \to Z}{\PiT{y : Y}{\SigmaT{z : Z}{\PiT{x : X}{f\,x} = z}}} \\
    &\equiv
      \SigmaT{f : X \to Z}{(Y \to \const f)}.&&\qedhere
  \end{align*}
\end{proof}

\begin{theorem}[\flink{Theorem-3-16} Leibniz adjunction]\label{Leibniz-adjunction-main-thm:alt}
  We have natural equivalences
  \[
    \Fam((A,B) \lprod (X,Y);(C,D))
    \simeq \Fam((A,B) ; (X,Y) \lexp (C,D)).
  \]
  Moreover, the pushout-product and pullback-hom extend to wild functors on
  \(\Map\) such that we have natural equivalences
  \[
    \Map(f \lprod g,h) \simeq \Map(f,g\lexp h).
  \]
\end{theorem}
\begin{proof}
  We compose the following simple equivalences, where the third equivalence comes from \cref{join-and-constancy}, to get the desired equivalence:
  \begin{align*}
    &\Fam((A,B) \lprod (X,Y);(C,D)) \\
    &\simeq
    \SigmaT{m : A \times X \to C}{%
      \PiT{(a,x) : A \times X}{\pa*{B(a) \join Y(x) \to D(m(a,x))}}} \\
    &\simeq
      \PiT{a : A}{%
      \SigmaT{k : X \to C}{%
      \PiT{x : X}{\pa*{B(a) \join Y(x) \to D(k\,x)}}}} \\
    &\simeq
      \PiT{a : A}{%
      \SigmaT{k : X \to C}{%
      \PiT{x : X}{\SigmaT{f : {Y\,x} \to D(k\,x)}(B\,a \to \const f)}}} \\
    &\simeq
      \PiT{a : A}{%
      \SigmaT{(k,\kappa) : \Fam((X,Y);(C,D))}{%
      \PiT{x : X}{\pa*{B\,a \to \const(\kappa\,x)}}}} \\
    &\simeq
      \PiT{a : A}{%
      \SigmaT{(k,\kappa) : \Fam((X,Y);(C,D))}{%
      \pa*{B\,a \to \PiT{x : X}{\const(\kappa\,x)}}}} \\
    &\simeq \Fam((A,B);(X,Y) \lexp (C,D)).
  \end{align*}

  The naturality can be checked directly.

  The second claim then follows via the isomorphism of
  \cref{Map-Fam-equivalent}, making use of
  \cref{preservation-of-Leibniz-exponential,preservation-of-Leibniz-product}.
  For instance, we can define the functorial action
  of \(\lprod\) as the composite
  \begin{align*}
    \Map(f,g) \simeq \Fam(\chi\,f,\chi\,g)
    \xrightarrow{{\lprod}}
    &\Fam(\chi\,f \lprod \chi\,h , \chi\,g  \lprod \chi\,h) \\
    \simeq
    &\Fam(\chi(f \lprod h),\chi(g \lprod h)) 
    \simeq 
              \Map(f \lprod h , g \lprod h). \qedhere
  \end{align*}
\end{proof}

This gives us an internal version of Riehl and Verity's adjunction observation \cite[Lemma 4.10]{riehl2014theory} for the wild category of types:

\begin{corollary}[\flink{Corollary-3-17}]\label{leibniz-internal-adjunction}
  For maps \(i\), \(j\) and \(f\), the objects $(i \Ltimes j) \Lexp f$ and $i \Lexp (j \Lexp f)$ are equal.
\end{corollary}
\begin{proof}
  This follows abstractly, using the fact that the Yoneda embedding is fully
  faithful and that such functors reflect isomorphisms. Indeed, for an arbitrary
  map \(m\), we have natural equivalences
  \begin{align*}
    \Map(m, i \Lexp (j \Lexp f)) 
    &\simeq
      \Map(m \Ltimes i , j \Lexp f)
    &&\text{(by the adjunction)}
    \\
    &\simeq
      \Map((m \Ltimes i) \Ltimes j , f)
    &&\text{(by the adjunction)} \\
    &\simeq
      \Map(m \Ltimes (i \Ltimes j) , f)
    &&\text{(by the associativity of \(\Ltimes\), \cref{Leibniz-product-associative-commutative})} \\
    &\simeq
      \Map(m , (i \Ltimes j) \Lexp f)
    &&\text{(by the adjunction)}.
  \end{align*}

  Alternatively, as recorded in the formalization, one can establish the
  relevant isomorphism in \(\Fam\) directly and then transfer it to \(\Map\).
\end{proof}

\section{Application: Orthogonality and Simplicial Type Theory}

In the introduction, we discussed the problem of finding a unique extension of a
map \(\Lambda^2_1 \to X\), which picks out two composable morphisms in \(X\),
along the inclusion \(\Lambda^2_1 \hookrightarrow \Delta^2\), resulting in a map
\(\Delta^2 \to X\) that, in particular, yields the composition of the two
composable morphisms.
This situation can be pictured as a (degenerate) commutative square
\[
  \begin{tikzcd}
    \Lambda^2_1 \ar[r] \ar[d, hookrightarrow] \ar[r]  & X \ar[d] \\
    \Delta^2 \ar[ur, dashed] \ar[r]  & 1
  \end{tikzcd}
\]
where we look for a unique (dashed) map that makes the resulting two triangles
commute.
It is fruitful to generalize this situation to an arbitrary map \(i : A \to B\)
on the left of the square and an arbitrary map \(f : Y \to X\) on the right (as
in \eqref{eq:commutative-square} below), in which case we say that \(i\) is
orthogonal to \(f\) when all such squares admit unique diagonals.

In \cref{sec:orthogonality} we study this abstract situation and then in
\cref{sec:stt} we apply the results in the context of simplicial type theory.

\subsection{Orthogonality}\label{sec:orthogonality}
Given the solid commuting square
\begin{equation}\label{eq:commutative-square}
	\begin{tikzcd}
		A \ar[r] \ar[d,"i"']\ar[r,"u"]  & Y \ar[d,"f"]\\
		B \ar[ur,dashed,"d"]\ar[r,"v"]  & X
	\end{tikzcd}
\end{equation}
we can consider the type of diagonal lifts, i.e.\ the type of functions $d : B \to Y$ that make the triangles commute.
Because we are in the wild categorical setting of untruncated types, we have to take the higher cells into account.

\begin{definition}[\flink{Definition-4-1} Diagonal fillers of a square; {\cite[Def.~1.43]{rijke2020modalities}}]
	A~\emph{commuting square} is an arrow in $\Map$.
	Given a commuting square~$\Gamma$, its \emph{type of diagonal fillers}
        consists of a function, proofs of commutativity of the triangles, and a proof that the triangles compose to the square.
\end{definition}

Unfolding the definition, a commuting square consists of four types $A$, $B$, $X$, $Y$, and four functions $i$, $f$, $u$, $v$, as shown in \eqref{eq:commutative-square}, together with a proof of commutativity, $\gamma : v \circ i = f \circ u$.
A diagonal filler consists of quadruples $(d,\alpha,\beta,p)$, where:
\begin{itemize}
	\item $d : B \to Y$ is a function,
	\item $\alpha : u = {d \circ i}$ is a commutativity witness of the upper triangle,
	\item $\beta : {f \circ d} = v$ witnesses commutativity of the lower triangle,
	\item and $p : \ap_{f \circ \_} \alpha \cdot \ap_{\_ \circ i} \beta = \gamma$ says that $\alpha$ and $\beta$ compose to~$\gamma$.
\end{itemize}

\begin{definition}[\flink{Definition-4-2} Orthogonality, $i \orth f$]
	Let $i : A \to B$ and $f : Y \to X$ be functions.
	We say that $i$ is \emph{left orthogonal} to $f$ (and $f$ is \emph{right orthogonal} to $i$), written $i \orth f$, if every commutative square as in~\eqref{eq:commutative-square} has a contractible type of diagonal fillers.
\end{definition}

As emphasized by Riehl in a talk~\cite{RiehlContractibility}, contractibility
is the homotopically correct way of expressing uniqueness, so that we read the
above definition as ``every (such) commutative square has a unique filler''.
In classical homotopy theory, asking for unique fillers would be much too
strong; one wants filler to be unique \emph{up to homotopy} only, otherwise all
higher structure would be degenerate.
This is the key advantage of homotopy type theory where all constructions work up
to homotopy (viz.\ up to propositional identities).

\begin{lemma}[\flink{Lemma-4-3}]\label{lem:orth-vs-Lexp}
	For all maps $i$ and $f$, we have $i \orth f$ if and only if $i \Lexp f$ is an equivalence.
\end{lemma}
\begin{proof}
	Given $i$ and $f$, the type of squares is given by \(\Map(i,f)\); this is precisely the codomain of $i \Lexp f$.
	Every square having contractible type of fillers is equivalent to asking every fiber of $i \Lexp f$ to be contractible, which means that $i \Lexp f$ is an equivalence.
\end{proof}

\begin{lemma}[\flink{Lemma-4-4}]\label{lem:equivs-are-orth}
	Equivalences are left- and right-orthogonal to any map. That is, if $e$ is an equivalence and $g$ an arbitrary function, then $e \orth g$ and $g \orth e$.
\end{lemma}
\begin{proof}
  Let \(e : A \to B\) be an equivalence and \(g : X \to Y\) an arbitrary map.
  Recalling~\eqref{Leibniz-exponential}, we consider the diagram
  \[
    \begin{tikzcd}[column sep=.5cm]
      X^B
      \ar[dr,"e \lexp g"]
      \ar[drr,bend left=20,"X^e"]
      \ar[ddr,bend right,"g^B"]
      \\
      & X^A \times_{B^A} Y^B \pbcorner
      \ar[d] \ar[r,"t"]
      & X^A \ar[d,"g^A"]
      \\
      & Y^B \ar[r,"Y^e"]
      & Y^A
    \end{tikzcd}
  \]
  Since \(e\) is an equivalence, so is the precomposition map \(Y^e\) and hence
  its pullback~\(t\).
  The postcomposition map \(X^e\) is also an equivalence, so that
  \(e \Lexp g\) is as well by 2-out-of-3.

	If $e$ and $g$ are swapped, the vertical, instead of the horizontal, maps are equivalences.
\end{proof}

\begin{lemma}[\flink{Lemma-4-5}]\label{lem:left-orth-pushout-products}
	Left orthogonal maps are closed under pushout-products with arbitrary functions.
	That is, if $i$ is left orthogonal to $f$ and $j$ is any map, then $i \Ltimes j$ is left orthogonal to $f$.
\end{lemma}
\begin{proof}
	By \cref{lem:orth-vs-Lexp}, we can assume that $i \Lexp f$ is an equivalence, and we have to show that $(i \Ltimes j) \Lexp f$ is an equivalence.
	We have $(i \Ltimes j) \Lexp f = (j \Ltimes i) \Lexp f$ by commutativity of~$\Ltimes$ (\cref{Leibniz-product-associative-commutative}), and the latter expression equals $j \Lexp (i \Lexp f)$
	by \cref{leibniz-internal-adjunction}, which is an equivalence by \cref{lem:equivs-are-orth}.
\end{proof}

Finally, we want to show that left orthogonal maps are closed under retracts.
Here, ``retract'' is to be understood in the usual categorical sense in the wild
category \(\Map\), i.e.\ we have commutative squares
\[
  \begin{tikzcd}
    \bullet \ar[d,"j"] \ar[r,"s"] & \bullet \ar[d,"i"] \ar[r,"r"]
    \ar[dl,Rightarrow,"\alpha"',shorten=2mm]
    & \bullet \ar[d,"j"]
    \ar[dl,Rightarrow,"\beta"',shorten=2mm]
    \\
    \bullet \ar[r,"s'"] & \bullet \ar[r,"r'"]
    & \bullet
  \end{tikzcd}
\]
that compose to \((\id,\id,\refl)\).
Unravelling this definition, we see that it is a reformulation of
\cite[Definition~4.7.2]{HoTTBook}.

\begin{lemma}[\flink{Lemma-4-6}]\label{lem:left-orth-retracts}
	If $j$ is a retract of $i$, and $i \orth f$, then $j \orth f$.
\end{lemma}
\begin{proof}
  Suppose that \(j\) is a retract of \(i\) and that \(i \Lexp f\) is an
  equivalence. We have to prove that \(j \Lexp f\) is also an equivalence.
  Since equivalences are closed under retracts~\cite[Theorem~4.7.4]{HoTTBook},
  it suffices to show that \(j \Lexp f\) is a retract of
  \(i \Lexp f\).
  But this holds because \((-) \Lexp f\) is a wild functor (by
  \cref{Leibniz-adjunction-main-thm:alt}), and any wild functor preserves retracts.
\end{proof}

\subsection{Simplicial Type Theory}\label{sec:stt}

Simplicial type theory was introduced by Riehl and Shulman~\cite{RS2017};
we follow the ``internal'' approach of \cite{GWB2024,GWB2025}, generalized (i.e.\ with weaker assumptions) as described in \cref{sec:foundational-setup}.

\subsubsection*{Assumption (\flink{Assumption})}
We assume that we are given a set $I$ (the ``interval'') together with $0, 1: I$ and binary operations $\intmeet, \intjoin$ that turn $I$ into a bounded distributive lattice.
Explicitly, the equations are:

\begin{align*}
  x \intmeet x &= x
  &
    x \intjoin x &= x
  &
    x \intmeet (x \intjoin y) &= x
  \\
  x \intmeet y &= y \intmeet x
  &
    x \intjoin y &= y \intjoin x
  &
    x \intjoin (x \intmeet y) &= x
  \\
  x \intmeet (y \intmeet z) &= (x \intmeet y) \intmeet z
  &
    x \intjoin (y \intjoin z) &= (x \intjoin y) \intjoin z
  &
    x \intmeet (y \intjoin z) &= (x \intmeet y) \intjoin (x \intmeet z)
  \\
  x \intmeet 1 &= x
  & x \intjoin 0 &= x
  \\
\end{align*}

As is standard, we write $x \leq y$ (and $y \geq x$) for $x \intmeet y = x$.
This turns $\leq$ into a reflexive and transitive binary order on $I$, with $0$~and~$1$ as least and greatest element.%
\footnote{The alternative definition of $x \leq y$ as $x \intjoin y = y$ is equivalent.}

An element of the $n$-simplex is given by a decreasing sequence $x_1 \geq x_2 \geq \ldots \geq x_n$ of interval variables.
For the implementation of the concept in type theory, it is convenient to add an additional $1$ in the beginning and $0$ at the end.
So, for every $n : \N$, we define the type
\[
  \Delta^n \defeq
  \textstyle\sum{(x : \Fin {(n+2)} \to I)}\,(x_0 = 1) \wedge (x_{n+1} = 0) \wedge (\forall i. x_i \geq x_{i+1}).
\]
We apply the first projection implicitly which means that, for $x : \Delta^n$ and $i : \Fin {(n+2)}$, we simply write $x_i : I$.
As a simple example, taking \(n = 2\)
, we note that
\(
  \Delta^2 \simeq \SigmaT{x,y : I}{x \geq y},
\)
so that, plotting \(x,y\) along their usual axes, we can picture \(\Delta^2\) as in~\eqref{Delta-2-picture}.

\noindent\begin{minipage}{.3\linewidth}
\begin{equation}\label{Delta-2-picture}
  \begin{tikzcd}[column sep=.3cm, row sep=.3cm,%
    execute at end picture={
      \begin{pgfonlayer}{background}
        \foreach \Name in  {A,B,C}
        {\coordinate (\Name) at (\Name.center);}
        \fill[dashed,gray!40]
        (A) -- (B) -- (C) -- cycle;
      \end{pgfonlayer}
    }
    ]
    & |[alias=A]|(1,1)  \\
    |[alias=C]|(0,0) \ar[ur,shorten=-5pt] \ar[r,shorten=-3pt] & |[alias=B]|(1,0) \ar[u,shorten=-4pt]
  \end{tikzcd}
\end{equation}
\end{minipage}%
\begin{minipage}{.3\linewidth}
  \begin{equation*}\end{equation*} 
\end{minipage}
\begin{minipage}{.3\linewidth}
  \begin{equation}\label{Lambda-2-1-picture}
  \begin{tikzcd}[column sep=.3cm, row sep=.3cm]
    & (1,1)  \\
    (0,0) \ar[r,shorten=-3pt] & (1,0) \ar[u,shorten=-4pt]
  \end{tikzcd}
\end{equation}
\end{minipage}

Generalizing from the construction of low-dimensional horns in \cite{RS2017,GWB2024,GWB2025}, we define the general \((n,k)\)-horn \(\Lambda^n_k\), for $n : \N$ and $k : \Fin {(n+1)}$,
as a subtype of $\Delta^n$ by
\begin{equation*}
	\Lambda^n_k \defeq \SigmaT{x : \Delta^n}{\exists j : \Fin {(n+1)}. j \not= k \wedge x_j = x_{j+1}}.
\end{equation*}
We write $\lambda^n_k : \Lambda^n_k \to \Delta^n$ for the canonical subtype inclusion, i.e.\ for the first projection $\fst$.

Taking \(n = 2\) again and \(k = 1\), we note that
\[
  \Lambda^2_1
  \simeq \SigmaT{x : \Delta^2}{(x_0 = x_1) \vee (x_2 = x_3)}
  \simeq
  \SigmaT{x,y:I}{(x = 1) \vee (y = 0)},
\]
so that it is pictured as in~\eqref{Lambda-2-1-picture} and \(\lambda^2_1\) is
the canonical inclusion of \eqref{Lambda-2-1-picture} into
\eqref{Delta-2-picture}.

\begin{remark}
	The horn $\Lambda^n_k$ can be pictured as the standard simplex~$\Delta^n$ where the ``inner'' part (the single non-degenerated $n$\nobreakdash-cell) and the $k$-th face are missing.
	This is mirrored in the above definition: If $x$ satisfies $x_j = x_{j+1}$, then the tuple is an element of the $j$-th face.
\end{remark}

Following \cite{RS2017}, we now recall Segal types and inner anodyne maps.
\begin{definition}[\flink{STT.Definition-4-8} Segal types and inner anodyne maps]
	A type $X$ is a \emph{Segal type} if the map $X \to 1$ is right orthogonal to the horn inclusion $\lambda^2_1$.
	A function $i : A \to B$ is \emph{inner anodyne} if it is left orthogonal to every Segal type.
\end{definition}

A Segal type should be thought of as a higher category (cf.\ the introduction).
While some of the theory of Segal types can be developed in a fairly minimal
setting, as we show here, interesting examples of Segal types appear to need
additional features, notably modalities, see \cite{GWB2024,GWB2025} for more
details.

The intuition for inner anodyne maps is as follows: if \(f : A \to B\) is inner
anodyne, then \(A\) is thought of as a smaller shape and \(B\) as a larger shape
where a shape is something that can carry information by mapping from it into
another type. The map being inner anodyne means that, as long as the information
is carried in Segal types, the additional information that \(B\) carries can be
uniquely reconstructed from the information that \(A\) carries.
Matching this intuition, the maps $\lambda^3_1$ and $\lambda^3_2$ are inner
anodyne, by a result of Riehl and Shulman~\cite[Proposition 5.21]{RS2017}.

We generalize this result to all inner horn inclusions, i.e.\ all horn inclusions $\lambda^n_k$ where $k$ is strictly between $0$ and $n$.
This is a version of a result by Lurie~\cite[Corollary 2.3.2.2]{LurieHTT} in simplicial type theory.
\begin{theorem}[\flink{STT.Theorem-4-9}]\label{thm:inner-horns-are-anodyne}
	For $n : \N$ and $k$ with $0 < k < n$, the inner horn inclusion $\lambda^n_k$ is inner anodyne.
\end{theorem}
\begin{proof}
	We show that $\lambda^n_k$ is a retract of the pushout-product $\lambda^n_k \Ltimes \lambda^2_1$.
	Using \cref{lem:left-orth-pushout-products,lem:left-orth-retracts}, this is enough to conclude that any map that is right orthogonal to $\lambda^2_1$ is also right orthogonal to~$\lambda^n_k$.

	To prove the claimed retract property, we need to define maps
        \[
          s: \lambda^n_k \to \lambda^n_k \Ltimes \lambda^2_1 \quad\text{and}\quad
          r: \lambda^n_k \Ltimes \lambda^2_1 \to \lambda^n_k
        \]
	and show that they compose to the identity:
	\begin{equation}\label{eq:retract-diagram}
		  \begin{tikzcd}
			\Lambda^n_k \ar[d,"\lambda^n_k"] \ar[r,"s_\mathsf{dom}"] &
			(\Delta^n \times \Lambda^2_1) +_{\Lambda^n_k \times \Lambda^2_1} (\Lambda^n_k \times \Delta^2)
			\ar[d,"\lambda^n_k \Ltimes \lambda^2_1"] \ar[r,"r_\mathsf{dom}"] & \Lambda^n_k \ar[d,"\lambda^n_k"] \\
			\Delta^n \ar[r,"s_\mathsf{cod}"] & \Delta^n \times \Delta^2 \ar[r,"r_\mathsf{cod}"] & \Delta^n
		\end{tikzcd}
	\end{equation}

	We start by defining $s$ by first giving its domain and codomain parts; we omit the propositional information.
	\[
		s_\mathsf{dom}(x) \defeq \mathsf{inr}(x; (x_{k},x_{k+1})) \quad\text{and}\quad
		s_\mathsf{cod}(x) \defeq (x; (x_{k},x_{k+1})),
	\]
	This makes the left square commute judgmentally, and thus defines the section $s$.

	The retraction $r$ is almost as easy to define.
	We start by giving the codomain part:
	\begin{align*}
		&r_\mathsf{cod}(x,y) \defeq \lambda i.
		\begin{cases}
			x_i \intjoin y_1, & \text{if } i \leq k,\\
			x_i \intmeet y_2,  & \text{if } i > k.
		\end{cases}
	\end{align*}
	Since the right vertical map $\lambda^n_k$ is an embedding, this fully determines the domain part $r_\mathsf{dom}$.
	To state it explicitly, it is a map out of a pushout and can thus be described as a pair of two maps,
	\[
          r_\mathsf{dom}^1 : \Delta^n \times \Lambda^2_1 \to \Lambda^n_k \quad\text{and}\quad
          r_\mathsf{dom}^2 : \Lambda^n_k \times \Delta^2 \to \Lambda^n_k,
	\]
	and both components are given by the same equation as $r_\mathsf{cod}$; we only need to check that the propositional conditions of the output can be constructed from the input.
	This is (almost) automatic for~$r_\mathsf{dom}^2$, where the missing face $k$ remains the same.
	For~$r_\mathsf{dom}^1(x,y)$, we need to check the face of the input in the second component; we either have $1 = y_1$ or $y_2 = 0$.
	In these two cases, we choose $j = 0$ and $j = n$ respectively, and both choices satisfy $j \not= k$ because we assumed $0 < k < n$.
	The two components $r_\mathsf{dom}^1$ and $r_\mathsf{dom}^2$ define a map out of the pushout because their definitions only differ in the propositional part.
	The square determined by $(r_\mathsf{dom},r_\mathsf{cod})$ again commutes almost judgmentally.

	Finally, we are required to check that $s$ and $r$ compose to the identity on $\lambda^n_k$.
	We verify this for the codomain part, i.e.\  the lower horizontal row in \eqref{eq:retract-diagram}.
	Given an input $x : \Delta^n$, the $i$-th component of $(r_\mathsf{cod} \circ s_\mathsf{cod})(x))$ is $x_i \intjoin x_{k}$ if $i \leq k$, and $x_i \intmeet x_{k+1}$ if $i > k$; both expressions are equal to $x_i$.
	The domain part follows automatically as an equation in a subtype of what we just checked; the composition of the two squares is trivial as the codomain~$\Delta^n$ is a~set.
\end{proof}

\begin{remark}
	While it can be proved that the pushout-product $\lambda^n_k \Ltimes \lambda^2_1$ is a map between sets, the argument does not require this.
	Just like \cite{RS2017}, we use a version of Joyal's lemma \cite[Proposition 2.3.2.1]{LurieHTT}, which is based on pushouts and retracts in the category of sets, while all our results are about the corresponding notions in the wild category of types.
	In homotopy type theory, the canonical inclusion of the univalent 1-category of sets into the wild category of types does not in general preserve pushouts (i.e.\ not all set-pushout squares are homotopy pushouts); fortunately, pushouts along embeddings (e.g.\ the horn inclusions) are preserved, which explains why no mismatch between the different notions of pushout occurs.
\end{remark}

\begin{remark}[\flink{STT.Remark-4-11}]
	Of course, the proof of \cref{thm:inner-horns-are-anodyne} also shows that Segal \emph{fibrations} (instead of Segal types) satisfy the lifting condition with respect to all inner horn inclusions: If a map ${p : Y \to X}$ is right orthogonal to $\lambda^2_1$, then it is also right orthogonal to all other inner horn inclusion.
\end{remark}

\section{Conclusions}

As an application of our first key contribution, the Leibniz adjunction in homotopy type theory (\cref{Leibniz-adjunction-main-thm:alt}),
we have shown that all inner horn inclusions are inner anodyne (\cref{thm:inner-horns-are-anodyne}) in simplicial type theory with an internal interval.
In contrast, in the setting of Riehl and Shulman~\cite{RS2017}, even the statement itself would necessarily be meta-theoretic, as the quantification over all inner horns would have to happen on the meta-level.
This alone would make a full formalization tricky.

However, \emph{if} one works in a framework in which such a meta-theoretic quantification is feasible, and where working with meta-theoretic statements is convenient, then the equivalent version of our result \cref{thm:inner-horns-are-anodyne} is actually easier to prove.
The reason for this is that definitional equalities are proof-irrelevant and the meta-theoretic level is therefore ``set-based'' (in the sense of homotopy type theory).
Thus, the required version of the Leibniz adjunction is the 1-categorical one.
Compared to our \cref{Leibniz-adjunction-main-thm:alt}, the 1-categorical Leibniz adjunction \cite[Prop.~5.20]{RS2017}
does not necessitate the tracking of coherences and thus avoids what
represented the overwhelming majority of the challenge we encountered in \cref{sec:Leibniz-adjunction}.
A~framework in which working with such meta-theoretic statements is possible was
suggested by two of the authors~\cite{KdJ_representing} using a two-level type
theory~\cite{2LTT}.

While both the original approach with a meta-theoretic interval and the approach with an internal interval type have their advantages,
an important feature of the latter is that it allowed us to present a complete formalization in Agda.
Given the technical difficulties associated with the Leibniz construction, together with the combinatorial complexity of Joyal's lemma and the resulting need for careful index tracking, we regard this formalization as a useful tool for verifying the correctness of our arguments.

\section*{Declaration of generative AI use}
No generative AI was used for conducting this research or in preparation of this
paper and its accompanying formalization.


\bibliography{references.bib}

\end{document}